\documentclass[11pt]{amsart}
\def\ex{\textit{E}}
\def\var{\text{Var}}
\def\pr{\textit{P}}
\def\bi{\bigskip\noindent}
\def\a{\alpha}
\def\be{\beta}
\def\ga{\gamma}
\def\la{\lambda}
\def\La{\Lambda}

\def\bi{\bigskip\noindent}
\usepackage{graphicx}
\usepackage{latexsym}
\usepackage{amsfonts,amsmath,amssymb}
\newtheorem{theorem}{Theorem}[section]

\newtheorem{lemma}[theorem]{Lemma}

\newtheorem{corollary}[theorem]{Corollary}

\def\ex{\textit{E}}
\def\bi{\bigskip\noindent}

\date{\today}
\begin{document}

\title[Graphical partitions]{Asymptotic joint distribution of the extremities of a random Young diagram and
enumeration of graphical partitions}
\author{Boris Pittel}
\address{Department of Mathematics, The Ohio State University, $231$ West $18$-th Avenue,
Columbus, Ohio $43210-1175$ (USA)}
\email{bgp@math.ohio-state.edu}

\noindent
\keywords
{integer partition, Young diagram, tallest column, longest rows, enumeration, 
graphical partition, Macdonald/Wilf conjectures, convergence rate}
\subjclass[2010]{ 05A15, 05A16, 05A17, 11P82, 60C05, 60F10}
\begin{abstract} An integer partition of $n$ is a decreasing sequence of positive integers that add up to $[n]$. Back in $1979$ Macdonald posed a question about the limit value of the probability
that two partitions chosen uniformly at random, and independently of each other, are comparable
in terms of the dominance order. In $1982$ Wilf conjectured that the uniformly random partition
is a size-ordered degree sequence of a simple graph with the limit probability $0$. In $1997$ we 
showed that in both, seemingly unrelated, cases the limit probabilities are indeed zero, but
our method left open the problem of convergence rates. The main result in this paper is that each of the probabilities is $e^{-0.11\log n/\log\log n}$, at most. A key element of the argument is
a local limit theorem, with convergence rate,  for the {\it joint\/} distribution of the $[n^{1/4-\varepsilon}]$ tallest columns and the $[n^{1/4-\varepsilon}]$ longest rows of the Young diagram representing the random partition.
\end{abstract}
\maketitle
\section{Introduction and main results} A weakly decreasing sequence $\la=(\la_1,\dots,\la_m)$,
$m=m(\la)\ge 1$, 
of positive integers is called a partition of a positive integer $n$ into $m$ parts if $\la_1+\cdots+\la_m=n$. We will denote the set of all such partitions $\la$ by $\Omega_n$. It is customary to
visualize a partition $\la$ as a (Young-Ferrers) diagram formed by $n$ unit squares, with the 
columns of decreasing heights $\la_1,\dots,\la_m$. We will use the same
letter $\la$ for the diagram representing the partition $\la$. Introduce the positive integers
\[
\la'_i=\bigl|\{1\le j\le m(\la): \la_j\ge i\}\bigr|,\quad 1\le i\le \la_1;
\]
so $\la'_i$ is the number of parts in the partition $\la$ that are $i$, at least. Clearly $\la'_i$ decrease
and add up to $n$; so $\la':=(\la'_1,\dots,\la'_{m'})$, ($m'=\la_1$), is a partition of $[n]$, usually
referred to as the conjugate to $\la$.

The dominance order on the set $\Omega_n$ is a partial order 
$\preceq$ defined as follows. For  $\la,\mu\in \Omega_n$,  we write $\la\preceq\mu$ if 
\begin{equation}\label{dom}
\sum_{j=1}^i\la_j\le \sum_{j=1}^i\mu_j,\quad i\ge 1;
\end{equation}
(by definition $\la_j=0$ for $i>m(\la)$, $\mu_j=0$ for $j>m(\mu)$). Under $\preceq$, $\Omega_n$
is a lattice. Brylawski \cite{Bry} demonstrated how ubiquitous this lattice is. For instance,
Gale-Ryser theorem (Gale \cite{Gale} and Ryser \cite{Ryser}, Brualdi and Ryser \cite{BruRys}) asserts:  given two decreasing 
positive tuples $\a=(\a_1,\a_2, \dots,\a_r)$, $\be=(\be_1,\be_2,\dots,\be_s)$, 
there exists a {\it bipartite\/}
graph on a vertex set $(X,Y)$, $|X|=r$, $|Y|=s$, such $\a$ and $\be$ are the size-ordered 
degree sequences of vertices in $X$ and $Y$ respectively, iff $\sum_{t=1}^r\a_t=\sum_{t=1}^s\be_t$ and $\a\preceq\be'$. The lattice $\Omega_n$ is also at the core of the classic description of the irreducible representations of the symmetric group  $S_n$, see Diaconis \cite{Dia}, Macdonald \cite{Mac}, Sagan \cite{Sag}, for instance.

Soon after \cite{Gale}, \cite{Ryser}, Erd\H os and Gallai \cite{ErdGal} found the necessary
and sufficient conditions a partition $\la\in \Omega_n$, ($n$ even), has to satisfy to be
{\it graphical\/}, i.e. to be a size-ordered degree sequence of a simple graph. According
to Nash-Williams (see Sierksma and Hoogeveen \cite{SieHoo} for the proof), the Erd\H os-Gallai conditions
are equivalent to 
\begin{equation}\label{Nash}
\sum_{j=1}^i\la'_j\ge \sum_{j=1}^i\la_j+i,\quad 1\le i\le D(\la);
\end{equation}
here $D(\la)$ is the size of the Durfee square of $\la$, i.e. the number of rows of the largest square inscribed into the Young diagram $\la$. Obvious differences notwithstanding,
 the Gale-Ryser conditions and the Nash-Williams 
conditions are undeniably similar.

Macdonald \cite{Mac} (Ch.1, Section 1, Example 18) posed a probabilistic question, which may
be formally interpreted as follows. Let $\la,\,\mu\in \Omega_n$ be chosen uniformly at random and independently of each other; does $\pr(\la\preceq \mu)$ approach $0$ as $n\to\infty$? This question had already been there in the $1979$ edition of \cite{Mac}. In $1982$
Wilf conjectured that $\lim\pr(\la\text{ is graphical})=0$; apparently he wasn't aware of Macdonald's
question.

In an attempt to prove Wilf's conjecture, Erd\H os and Richmond \cite{ErdRich} found an expression for the limiting probability that $\la$ satisfies the first $k$ conditions \eqref{Nash} as a 
$2k$-dimensional integral, thus reducing the problem to the question whether the integral's
value $c_k\to 0$ as $k\to\infty$. The authors also showed that $P_n:=\pr(\la\text{ is graphical})$
is of order $n^{-1/2}$ at least, meaning that if $P_n\to 0$, it does so rather slowly. 
Rousseau and Ali \cite{RouAli} demonstrated that $\lim_{k\to\infty}
c_k\le 1/4$ and $c_k\ge 2^{-2k}\binom{2k}{k}$; consequently $\limsup P_n\le
1/4$, and $c_k$ cannot approach $0$ faster than $k^{-1/2}$, {\it if indeed\/} $c_k\to 0$.

Barnes and Savage \cite{BarSav} discovered a recurrence-based algorithm for computing the 
total number of graphical partitions of $n$. They demonstrated that for $n$ ranging from $2$ to $220$ the fraction (probability) of graphical partitions steadily, but slowly, decreases from $0.5$ to $0.3503\dots$, which is still above the Rousseau-Ali limiting bound $0.25$. More recently, Kohnert \cite{Koh} derived a new recursion formula that allowed to compute the fraction of graphical partitions for $n$ up to $910$. For $n=910$, the fraction is $0.3264\dots$, still noticeably
exceeding $0.25$. Thus,  as $n$ runs from $220$ to $910$, the fraction decreases
by $0.025$ only.

In \cite{Pittel2} we proved the positive answer to Macdonald's question. The idea of the proof,
certainly inspired by \cite{ErdRich}, \cite{RouAli}, was to show that 
\[
\lim_{k\to\infty}\limsup_{n\to\infty}\pr\bigl(\la,\mu\text{ meet the first }k\text{ conditions in }
\eqref{dom}\bigr)=0.
\]
A key tool was a theorem on the limiting joint distribution of the $k$ largest parts of the random
$\la$ due to Fristedt \cite{Fristedt}. We also confirmed Wilf's conjecture by proving that, as Erd\H os and Richmond expected, $\lim c_k=0$. We did so via a slight modification of the proof for Macdonald's question. The proofs similarity is due to an implicit discovery in \cite{ErdRich} that
the limiting {\it joint\/} distribution of the $k$ largest parts in $\la$ and in its conjugate $\la'$ is the same as that of the largest $k$ parts in two {\it independent\/} partitions of $n$.

In both cases, we found a way to use Kolmogorov's $0-1$ law for the tail events of a sequence of {\it independent\/} random variables to show that the limiting probability in question cannot be
anything but $0$ or $1$. And then we used a central limit theorem to rule out the value $1$.
So our solution left open a fundamental question about the actual convergence rates for
both Macdonald's and Wilf's probabilities.

Our main result in this paper is that, for $n$ sufficiently large,  each of those probabilities  is 
\begin{equation}\label{thebound}
\exp\left(-\frac{0.11\log n}{\log\log n}\right),
\end{equation}
at most. Thus, the bound \eqref{thebound} is negligible compared to any negative power of $\log n$, but it approaches $0$ slower than any $n^{-a}$, $a>0$ being fixed. 

We compared  the values of this bound and the exact values of the fraction of the graphical 
partitions for $n=250,\,450,\,910$, computed in \cite{Koh}. For what it's worth, replacing $0.11$
with $\approx 0.326$, $\approx 0.321$, $\approx 0.315$, we get the actual numerical values of the fraction. Could it be that the expression \eqref{thebound} with a constant  close to $0.25$ replacing $0.11$ is an asymptotic formula for the fraction?

As a direct byproduct of our
proofs, \eqref{thebound} also bounds $\pr(\la\preceq\la')$, i.e. the probability that the random
partition $\la$ is both an in-degree sequence and an out-degree sequence, each being
size-ordered, of a {\it directed\/} graph.

As the first step, we prove that, for $k=[n^{\ga}]$ and $\ga<1/4$,  the {\it total variation distance\/} of the joint distribution of the $k$ tallest columns {\it and\/} the $k$ longest rows of the random diagram $\la$ from the distribution of the random tuple 
\begin{equation}\label{sumsofiid}
\left(\left\lceil \frac{n^{1/2}}{c}\log\frac{\tfrac{n^{1/2}}{c}}{\sum_{j=1}^i \mathcal E_j}\right\rceil,
\left\lceil \frac{n^{1/2}}{c}\log\frac{\tfrac{n^{1/2}}{c}}{\sum_{j=1}^i \mathcal E'_j}\right\rceil\right)_{1\le i\le k},
\quad c:=\frac{\pi}{\sqrt{6}},
\end{equation}
is at most $n^{-1/2+2\ga}(\log n)^3$. Here $\mathcal E_1,\dots, \mathcal E_k, \mathcal E_1',\dots, \mathcal E_k'$ are independent
copies of the exponential random variable $\mathcal E$, with $\pr(\mathcal E>x)=e^{-x}$, $x\ge 0$. 

We note that Fristedt \cite{Fristedt} proved
a convergence theorem---in terms of Prohorov distance, without an explicit convergence rate---for the $k=o(n^{1/4})$ tallest columns, or, by symmetry, for the $k$ longest rows, but not for the joint distribution. In \cite{Pittel2}
we already observed that Fristedt's limit theorem can be reformulated in the form \eqref{sumsofiid}.
It was this observation that led us to the argument based on Kolmogorov's $0-1$ law.

To show the convergence in terms of total variation distance we use saddle-point techniques for analysis of Cauchy integrals representing the counts of restricted partitions, which involve the generating function of restricted partitions and Freiman's estimate of the Euler generating  function of unrestricted partitions.

Using this approximation theorem we reduce, asymptotically, the Nash-Williams conditions \eqref{Nash} to
their counterpart that involves the sums 
\[
S_i=\sum_{j=1}^i \mathcal E_j,\quad S'_i=\sum_{j=1}^i \mathcal E'_j.
\] 
A series of Chernoff-type bounds allows us essentially to embed the resulting event  into an intersection of $\tfrac{\log k}{\log\log k}$ 
independent events, each of probability roughly $\pr(\mathcal N\ge -1/2)$, $\mathcal N$ being
the standard normal variable. This yields the bound for Wilf's $\pr(\la\text{ is graphical})$, and also
for $\pr(\la\preceq\la')$. As for the bound of Macdonals's $\pr(\la\preceq \mu)$, its proof is
the simplified version of that for Wilf's probability, since $\la$ and $\mu$ are exactly independent.

The decision to use the total variation distance in this paper 
%for measuring the distance between the distribution of the random variables in \eqref{sumsofiid} and the joint distribution of the $k$ largest columns and the $k$ longest rows in  $\la$
was certainly influenced by Arratia and Tavar\'{e} \cite{ArrTav}.  For a variety of the random partition problems, including
the integer partitions, they obtained a conceptual formula for that distance between the distributions of the parts counts for the problem in
question, and the certain {\it independent\/} variables that produce those counts, distribution-wise, upon  conditioning on their total
{\it weight\/} value. Based on this formula, they formulated a general conjecture as to when this distance would converge to zero. We confirmed this
conjecture for the random set partition in \cite{Pittel1.5} and for the random integer partition $\la$ in \cite{Pittel2}. We hasten to add that the Arratia-Tavar\'e model, with its central premise of conditioning on the total weight of the auxiliary independent variables,
does not cover the comparison problem in this paper.

\section{Joint distribution of the $k$ largest heights and $k$ largest widths of the random
diagram}

A partition of $n$ is visualized  as the diagram of area $n$ with column heights decreasing from
left to right. Let $p_n$ denote the total number of partitions of $n$, or equivalently the Young diagrams of area $n$, i.e.
$p_n=|\Omega_n|$.  It is well known since Euler that
\begin{equation}\label{p(q)=}
p(q):=\sum_{n\ge 1}q^n p_n=\prod_{j\ge 1}(1-q^j)^{-1},\quad |q|<1.
\end{equation}
Let $p_{n, r,s}$ denote the total number of the (restricted) diagrams,  those with the tallest column of height
$\le r$ and the longest (base) row of length $\le s$. It is also known that
\begin{equation}\label{prs(q)=}
\begin{aligned}
p_{r,s}(q):=&\,\sum_{n\ge 1}q^n p_{n,r,s}=\frac{\prod_{i=1}^{r+s}(1-q^i)}{\prod_{j=1}^r(1-q^j)
\prod_{k=1}^s(1-q^k)}\\
=&\,p(q)\cdot \prod_{j>r}(1-q^j)\cdot\prod_{k>s}(1-q^k)
\cdot\prod_{i>r+s}(1-q^i)^{-1},
\end{aligned}
\end{equation}
see Andrews \cite{Andrews}, Section 3.2. 
Hardy and Ramanujan \cite{HardyRam} used the Euler formula to find a series type formula for $p(n)$, whose simple corollary delivers
\begin{equation}\label{p(n)=}
p(n)=\frac{e^{\pi\sqrt{2n/3}}}{4\sqrt{3}n}\bigl(1+O(n^{-1/2})\bigr).
\end{equation}
This result can be obtained in a short way via a remarkably simple formula due to Freiman (see
Postnikov \cite{Postnikov}):
\begin{equation}\label{Frei}
\prod_{k\ge 1}(1-e^{-ku})^{-1}=\exp\left(\frac{\pi^2}{6u}+\frac{1}{2}\text{Log}\frac{u}{2\pi}+O(|u|)\right),
\end{equation}
uniformly for $u\to 0$ within a wedge $\{u:\text{Im}\,
u\le\varepsilon\text{Re}\, u, \text{Re}\,u>0\}$,
$\varepsilon >0$ being fixed, and $\text{Log}$ standing for the main branch of the logarithmic
function. (Freiman used \eqref{Frei} to obtain a weaker version of \eqref{p(n)=} with the
remainder term $O(n^{-1/4+\varepsilon})$.) Our aim is a sharp asymptotic formula 
for $p_{n,r,s}$ with $r,s$ of order $n^{1/2}\log n$. As a warm-up preparation, we derive \eqref{p(n)=}.
\begin{lemma}\label{|p(q)|<} Let $q=re^{i\theta}$, $0<r<1$, $\theta\in (-\pi,\pi]$. Then
\[
|p(q)|\le p(r)\exp\left(-\frac{\alpha r\theta^2}{(1-r)\bigl((1-r)^2+2r\alpha\theta^2\bigr)}\right),\quad \alpha:=2/\pi^2.
\]
\end{lemma}
\begin{proof} Using an inequality
\[
\left|\frac{1}{1-z}\right|\le \frac{1}{1-|z|}\,\exp\bigl(\text{Re}z-|z|),\quad (|z|<1),
\]
see \cite{Pittel1}, we obtain
\[
|p(re^{i\theta})|\le p(r)\exp\left(\sum_{j\ge 1}r^j(\cos(\theta j)-1)\right).
\]
Here
\begin{multline*}
\sum_{j\ge 1}r^j(\cos(\theta j)-1)=-(1-r)^{-1}+\text{Re}(1-re^{i\theta})^{-1}\\
=-\frac{r}{1-r}+\text{Re}\left(\sum_{j\ge 1}(re^{i\theta})^j\right)
=-\frac{r}{1-r}+\text{Re}\left(\frac{re^{i\theta}(1-re^{-i\theta})}{|1-re^{\theta}|^2}\right)\\
=-\frac{1+r}{1-r}\cdot\frac{r(1-\cos\theta)}{(1-r)^2+2r(1-\cos\theta)}\le-\frac{1+r}{1-r}\cdot\frac{\alpha r\theta^2}{(1-r)^2+2r\alpha \theta^2},
\end{multline*}
since $1-\cos\theta\ge \alpha\theta^2$ for $|\theta|\le \pi$.
\end{proof}
This Lemma and Freiman's formula yield \eqref{p(n)=} with a relatively little effort. First of all, by
Cauchy's integral formula,
\begin{equation}\label{Cauchy}
p_n=(2\pi i)^{-1}\!\!\!\!\!\oint\limits_{z=\rho e^{i\theta}\atop \theta\in (-\pi,\pi]}\!\!\!\!z^{-(n+1)}p(z)\,dz.
\end{equation}
Predictably, we want to choose $\rho$ close to the root of $(\rho^{-n}p(\rho))^\prime=0$, or setting
$\rho=e^{-\xi}$,
\[
\sum_{j\ge 1}\frac{j}{e^{j\xi}-1}=n,
\]
implying that 
\[
\xi^{-2}\left[\int_0^{\infty}\frac{y\,dy}{e^y-1}+O(\xi)\right]=n.
\]
Since the integral equals $\sum_{j\ge 1}1/j^2=\pi^2/6$, we select $\xi=c n^{-1/2}$,
$c=\tfrac{\pi}{\sqrt{6}}$. Break $(-\pi,\pi]$ in two parts, $[-n^{-\delta},n^{-\delta}]$ and 
$[-n^{-\delta},n^{\delta}]^c$, where $\delta\in (2/3,3/4)$. Since $\delta>1/2$, by  Lemma \ref{|p(q)|<}, 
and Freiman's formula,
\begin{align}
\left|\int_{|\theta|\ge n^{-\delta}}z^{-(n+1)}p(z)\,dz\right|\le&\,\rho^{-n}p(\rho)\int_{|\theta|\ge n^{-\delta}}
e^{-\alpha_1n^{3/2-2\delta}}\,d\theta\notag\\
=&\exp\left(bn^{1/2}-\alpha_2n^{3/2-2\delta}\right),\label{int,zlarge}
\end{align}
$\alpha_j>0$ being absolute constants. Consider $|\theta|\le n^{-\delta}$. Since $\delta>1/2$, we have $|\theta|=o(\xi)$. So we apply 
Freiman's formula for $u=\xi-i\theta$ and, using $\xi^2=n^{-1}\pi^2/6$,  easily obtain
\begin{align}
\log\frac{p(\rho e^{i\theta})}{(\rho e^{\i\theta})^n}=&\,b n^{1/2} +\frac{1}{2}\log\frac{c}{2\pi\sqrt{n}}-
\frac{1}{2\xi}\,i\theta-\theta^2 n^{3/2}\gamma_n\notag\\
&+i\theta^3 n^2\gamma_n^\prime+O\bigl(\theta^4 n^{5/2}+n^{-1/2}\bigl), \label{sum}
\end{align}
where $O(n^{-1/2})$ comes from $O(|u|)$ in Freiman's formula, and
\[
\gamma_n=c^{-1}\bigl(1+O(n^{-1/2})\bigr),\quad \gamma_n^\prime=O(1).
\]
(To be sure, there is also a term $O(n\theta^2)$ in \eqref{sum}, but it is absorbed by the 
big Oh-term already there.) For $|\theta|\le n^{-\delta}$, 
\[
|\theta|^3 n^2\le n^{2-3\delta}\to 0,\quad \theta^4 n^{5/2}\le n^{5/2-4\delta}\to 0,
\]
as $\delta>2/3$. Therefore
\begin{multline*}
\exp\left(-\frac{1}{2\xi}\,i\theta +i\theta^3 n^2\gamma_n^\prime+ O\bigl(\theta^4 n^{5/2}+n^{-1/2}\bigr)\right)\\
=1-\frac{1}{2\xi}\,i\theta +i\theta^3 n^2\gamma_n^\prime +O\bigl(\theta^4 n^{5/2}+n^{-1/2}\bigr).
\end{multline*}
Consequently, as $dz=izd\theta$,
\begin{multline}\label{int,zsmall}
\oint\limits_{|\theta|\le n^{-\delta}}\!\!\!z^{-(n+1)}p(z)\,dz
=i\exp\!\left(\!bn^{1/2}+\frac{1}{2}\log\frac{c}{2\pi\sqrt{n}}\right)\\
\times
\int\limits_{|\theta|\le n^{-\delta}}\!\!\!\exp\bigl(-\theta^2n^{3/2}\gamma_n\bigr)\!
\left(1-\frac{1}{2\xi}\,i\theta +i\theta^3 n^2\gamma_n^\prime +O\bigl(\theta^4 n^{5/2}+n^{-1/2}\bigr)\!\!\right)\,d\theta.
\end{multline}
Here
\begin{align*}
&\int\limits_{|\theta|\le n^{-\delta}}\!\!e^{-\theta^2n^{3/2}\gamma_n}\,d\theta=(\pi c/n^{3/2})^{1/2}
(1+o(n^{-1/2})),\\
&\int\limits_{|\theta|\le n^{-\delta}}\!\!e^{-\theta^2 n^{3/2}\gamma_n}\left(-\frac{1}{2\xi}\,i\theta +i\theta^3 n^2\gamma_n^\prime\right)\,d\theta=0,\\
&\int\limits_{|\theta|\le n^{-\delta}}\!\!e^{-\theta^2 n^{3/2}\gamma_n}\bigl(\theta^4 n^{5/2}+n^{-1/2}\bigr)\,d\theta
=O\bigl(n^{-3/4}n^{-1/2}\bigr).
\end{align*}
Using these equations together with \eqref{Cauchy}, \eqref{int,zlarge}, \eqref{int,zsmall}
we finish the proof of \eqref{p(n)=}.\\

Let us show how to modify the argument above to obtain an asymptotic formula for $p_{n,r,s}$
for $r,s\gg \tfrac{n^{1/2}}{4c}\log n$.
\begin{lemma}\label{pnrssim} Let $h=h(n)>0$, $w=w(n)>0$ be  
 such that  $h,w=O(n^{\beta})$, $\beta<1/4$. If $r$ and $s$ are the integer parts of
\[
\frac{n^{1/2}}{c}\log\frac{\tfrac{n^{1/2}}{c}}{h},\quad \frac{n^{1/2}}{c}\log\frac{\tfrac{n^{1/2}}{c}}{w}
\]
respectively, then 
\[
p_{n,r,s}=\frac{e^{\pi\sqrt{2n/3}}}{4\sqrt{3}n}e^{-h-w}\bigl(1+O(n^{-1/2}(h+w+1)^2)\bigr).
\]
\end{lemma}
\begin{proof} The number $p_{n,r,s}$ is given by the Cauchy integral formula, like \eqref{Cauchy},
with $p_{r,s}(q)$ instead of $p(q)$. We choose again the circle of radius  $\rho=e^{-\xi}$,
$\xi=c/n^{1/2}$. On this
circle, the product $\prod_{i>r+s}(1-q^i)^{-1}$ in the formula \eqref{prs(q)=} for $p_{r,s}(\rho e^{i\theta})$ is (uniformly) $\exp\left(O\left(\rho^{r+s}/(1-\rho)\right)\right)$, and
\[
\frac{\rho^{r+s}}{1-\rho}\sim\frac{\exp\left(-\log\tfrac{n}{c^2hw}\right)}{1-e^{-cn^{-1/2}}}=
O\bigl(n^{-1/2}hw\bigr)=O\bigl(n^{-1/2}(h^2+w^2)\bigr);
\]
so
\begin{equation}\label{last}
\prod_{i>r+s}(1-q^i)^{-1}= 1+O\bigl(n^{-1/2}(h^2+w^2)\bigr).
\end{equation}
As for two other products in \eqref{prs(q)=},
\begin{equation}\label{two}
\begin{aligned}
\prod_{j>r}(1-q^j)\prod_{k>s}(1-q^k)=&\,\exp\!\left[-\frac{q^r+q^s}{1-q}+O\!\left(\!\rho^r+\rho^s
+\frac{\rho^{2r}+\rho^{2s}}{1-\rho}\right)\right]\\
=&\,\exp\!\left(\!\!-\frac{q^r+q^s}{1-q}+O\bigl(n^{-1/2}(h^2+w^2)\bigr)\right).
\end{aligned}
\end{equation}
We use \eqref{last} and \eqref{two} to bound the contribution  to the contour integral
coming from $\theta$s not too close to $0$. To this end, evaluate first
\begin{multline*}
\text{Re}\,\left(\frac{e^{ir\theta}}{1-q}\right) -\frac{1}{1-\rho}
=\frac{\cos \theta r-\rho\cos ((r-1)\theta)}{1-2\rho\cos\theta+\rho^2}-\frac{1}{1-\rho}\\
=\frac{1}{(1-2\rho\cos\theta+\rho^2)(1-\rho)}\cdot\biggl[
(1-\rho)\bigl(\cos\theta r-\cos ((r-1)\theta)\bigr)\biggr.\\
\biggl.+(1-\rho)^2\bigl(\cos((r-1)\theta) -1\bigr)+2\rho(\cos\theta-1)\biggr].
\end{multline*}
Using $|\sin x|\le |x|$ and $\rho=e^{-c/n^{1/2}}$, the expression within the square brackets
is of order 
\[
\theta^2\bigl(n^{-1/2}r+n^{-1}r^2+1\bigr)=O(\theta^2\log^2n),
\]
and the denominator exceeds a constant factor times $n^{-1/2}\bigl(n^{-1}+\theta^2\bigr)$; so
\[
\text{Re}\,\left(\frac{e^{ir\theta}}{1-q} \right)-\frac{1}{1-\rho} =
O\left(\frac{\theta^2\log^2n}{n^{-1/2}\bigl(n^{-1}+\theta^2\bigr)}\right).
\]
Consequently, as $q=\rho e^{i\theta}$, $\rho^r\sim\tfrac{ch}{n^{1/2}}$ and $h=O(n^{\beta})$, we have
\begin{equation}\label{Rer}
\text{Re}\,\left(\frac{q^r}{1-q} \right)-\frac{\rho^r}{1-\rho} =
O\left(\frac{\theta^2 n^{\beta}\log^2n}{n^{-1}+\theta^2}\right),
\end{equation}
and the analogous estimate holds for $\tfrac{q^s}{1-q}$.

Let 
\[
|\theta|\ge n^{-\delta},\quad \delta\in (\max\{2/3, 1/2+\beta\},3/4). 
\]
The remainder term in
\eqref{Rer} is $O\bigl(\theta^2 n^{1+\beta}\log^2n\bigr)$. So using also
\eqref{last}, \eqref{two},  and then Lemma \ref{|p(q)|<} and $\beta<1/2$,  we obtain
\begin{align*}
|p_{r,s}(\rho e^{i\theta})|\le&\, 2|p(\rho e^{i\theta})|
\exp\left[-\frac{\rho^r+\rho^s}{1-\rho} +O\bigl(\theta^2 n^{1+\beta}\log^2n\bigr)\right]\\
\le&\, 2p(\rho)\exp\left(-\frac{\rho^r+\rho^s}{1-\rho} - \alpha_3 n^{3/2-2\delta}\right).
\end{align*}
Here
\[
\frac{\rho^r+\rho^s}{1-\rho}=h+w +O\bigl(n^{-1/2}(h+w)\bigr)=h+w+O\bigl(n^{-1/2+\beta}\bigr).
\]
So, as $-1/2+\beta< 3/2-2\delta$, 
\begin{align}
\left|\,\,\int\limits_{|\theta|\ge n^{-\delta}}\!\!\!\frac{p_{r,s}(z)}{z^{n+1}}\,dz\,\right|\le&\,4\pi\rho^{-n}p(\rho)
\exp\!\left(\!-\frac{\rho^r+\rho^s}{1-\rho}-\alpha_3n^{3/2-2\delta}\!\right) \notag\\
\le&\,\exp\!\left(bn^{1/2}-h-w -\alpha_4n^{3/2-2\delta}\right).\label{intprs,zlarge}
\end{align}

Let $|\theta|\le n^{-\delta}$. We need a sharp estimate for $\tfrac{q^r+q^s}{1-q}$ in \eqref{two}. An easy computation shows that, for $q=\rho e^{i\theta}$,
\begin{multline*}
\frac{q^r}{1-q}=\frac{\rho^r}{1-\rho}\left[ 1+i\theta\left(r+\frac{\rho}{1-\rho}\right)+O((r\theta)^2)\right]\\
=\frac{\rho^r}{1-\rho} +i\theta\frac{\rho^r}{1-\rho}\,\left(r+\frac{\rho}{1-\rho}\right)+
O\left(\frac{(r\theta)^2\rho^r}{1-\rho}\right).
\end{multline*}
Here
\[
\frac{\rho^r}{1-\rho}=(1+O(n^{-1/2}))h,
\]
and, since $r\sim \tfrac{n^{1/2}}{c}\log\tfrac{n^{1/2}/c}{h}$, $h=O(n^{\beta})$,  we have the bounds
\begin{align*}
|\theta| \frac{\max\{r,n^{1/2}\}\rho^r}{1-\rho}=&\,O(n^{-\delta}hn^{1/2}\log n)=
O\bigl(n^{1/2+\beta-\delta}\log n\bigr)\to 0,\\
\frac{(r\theta)^2\rho^r}{1-\rho}=&O(n^{1-2\delta}h\log^2n)=O\bigl(n^{1+\beta-2\delta}\log^2n)\to 0.
\end{align*}
Of course, we have the similar formulas for $q^s/(1-q)$. Therefore
\begin{align*}
\frac{q^r+q^s}{1-q}=&\,h+w +i\theta\frac{\rho^r}{1-\rho}\left(r+\frac{\rho}{1-\rho}\right)\\
&+O\left(\frac{(r\theta)^2\rho^r}{1-\rho}\right)+ O\bigl(n^{-1/2}(h+w)\bigr),
\end{align*}
with the $\theta$-dependent terms approaching zero uniformly for $\theta$ in question. 
Plugging this expression into \eqref{two} and using also \eqref{last},  we see that the contribution of the interval $[-n^{-\delta},n^{-\delta}]$ to the
contour integral representing $p_{n,r,s}$ is obtained via replacing the second factor
in the integrand over those $\theta$'s in \eqref{int,zsmall} with
\begin{multline*}
e^{-h-w}\biggl[1+i\theta\left(-\frac{1}{2\xi}+\frac{\rho^r}{1-\rho}\left(r+\frac{\rho}{1-\rho}\right)\right)
\biggr.\\
\biggl.+i\theta^3n^2\gamma_n' + O\bigl(\theta^4n^{5/2}+n^{-1/2}(h+w+1)^2\bigr)+
O\left(\frac{(r\theta)^2\rho^r}{1-\rho}\right)\biggr].
\end{multline*}
As in the case of $p_n$, the contributions of the $i\theta$-term and the $i\theta^3$-term to the resulting
integral are both zero, and we get
\begin{align*}
\oint\limits_{|\theta|\le n^{-\delta}}\frac{p_{r,s}(z)}{z^{n+1}}\,dz=&\,
i\exp\left(bn^{1/2}+\frac{1}{2}\log\frac{c}{2\pi\sqrt{n}} - h-w\right)\\
&\times\frac{(\pi c)^{1/2}}{n^{3/4}}\bigl(1+O(n^{-1}(h+w+1)^2)\bigr).
\end{align*}
Combining this formula with the bound  \eqref{intprs,zlarge}, and dividing by $2\pi i$, we complete the proof of Lemma \ref{pnrssim}.
\end{proof}

Let $\La=(\La_1,\La_2,\dots)$ denote the uniformly random partition of $n$,
visualized as the Ferrers diagram $\La$ with left-to-right ordered columns of decreasing heights
$\Lambda_1\ge \Lambda_2\ge\dots$. Let $\La'=(\La'_1,\La'_2,\dots)$ denote the 
partition (Ferrers diagram) conjugate to $\La$; so $\La'_j$ is the length of the $j$-th longest {\it row\/}
of $\La$. 

Lemma \ref{pnrssim} yields an asymptotic formula
for the joint distribution of $\La_1$, $\La'_2$, together with a convergence rate.
\begin{corollary}\label{La1,La'1} Introduce $H_1$, $W_1$ by setting
\[
\La_1=\frac{n^{1/2}}{c}\log\frac{\tfrac{n^{1/2}}{c}}{H_1},\quad \La'_1=\frac{n^{1/2}}{c}\log\frac{\tfrac{n^{1/2}}{c}}{W_1}.
\]
If $h,w=O(n^{\beta})$, $\beta<1/4$, then
\begin{equation*}
\pr(H_1\ge h,\,W_1\ge w)= e^{-h-w}\bigl(1+O(n^{-1/2}(h+w+1)^2)\bigr).
\end{equation*}
Informally, $H_1$ and $W_1$ are asymptotically independent, each exponentially distributed
with parameter $1$.
\end{corollary}
\begin{proof} Immediate from $\pr(H_1\ge h,\,W_1\ge w)=\tfrac{p_{n,r,s}}{p_n}$, with
\[
r:=\left\lceil\frac{n^{1/2}}{c}\log\frac{\tfrac{n^{1/2}}{c}}{h}\right\rceil,\,\,\,s:=\left\lceil\frac{n^{1/2}}{c}\log\frac{\tfrac{n^{1/2}}{c}}{w}\right\rceil,
\]
and Lemma \ref{pnrssim}.
\end{proof} 

The limit {\it marginal\/} distributions of $\La_1$ and $\La'_1$ were known, in the equivalent form, since
the pioneering work of Erd\H os and Lehner \cite{ErdLeh}. The novelty here is the asymptotic 
independence of $\La_1$ and $\La'_1$ {\it and\/} the explicit convergence rate. Fristedt
\cite{Fristedt} extended the Erd\H os-Lehner result considerably, by establishing the limit joint distribution
of the first $o(n^{1/4})$ largest parts of $\La$, thence, separately, the first  $o(n^{1/4})$ largest parts of $\La'$, but without an explicit convergence rate.
Our goal is to prove a counterpart of the Fristedt result for the  distribution of the first 
$k$ parts of $\La$ and {\it jointly\/} the first $k$ parts of $\La'$, together with an {\it explicit\/} convergence rate, for $k=[n^{\ga}]$, $\ga<1/4$.  We will use some of the techniques
from our studies of the random Young diagram \cite{Pittel1} and the random {\it solid\/} diagram  \cite{Pittel3}, with the added emphasis on the convergence rates as related to $k$, the number of the largest parts of $\La$ and $\La'$ that we focus on. \\ 

Let $\vec \la=(\la_1,\dots,\la_k)$, $\vec\la^\prime=(\la'_1,\dots,\la'_k)$ be such that $\la_j$ and
$\la'_j$ decrease, $\la_k>k$, $\la'_k>k$, and $\sum_j\la_j,\,\,\sum_j\la'_j\le n$. A diagram $\la$ of area $n$, with the leftmost $k$ tallest columns of height $\la_1,\dots,\la_k$, and the ``bottom-most'' $k$ longest rows of length
$\la'_1,\dots,\la'_k$, exists iff
\begin{equation}\label{nu=}
\nu:=n -\sum_{j=1}^k\la_j-\sum_{j=1}^k \la'_j+k^2>0.
\end{equation}
If we delete these $k$ columns and $k$ rows, we end up with a diagram of area $\nu$,
with the tallest column of height $r=\la_k-k$ at most, and the longest row of length $s=\la'_k-k$
at most. So introducing $p_n(\vec\la,\vec\la')$, the total number of the diagrams of area $n$ 
with parameters $\vec\la$, $\vec\la'$, we see that $p_n(\vec\la,\vec\la')=p_{\nu,r,s}$. To apply
Lemma \ref{pnrssim} to $p_{\nu,r,s}$, the parameters $\nu$,  $r=\la_k-k, s=\la'_k-k$ need to meet
the conditions of this Lemma, with $\nu$ playing the role of $n$, of course.  This observation
coupled  with the statement of Corollary \ref{La1,La'1}, and Fristedt's result for the
first $o(n^{1/4})$ parts of the random partition is our motivation for focusing on $k=[n^{\gamma}]$,
$\gamma<1/4$, and the integers $\la_j$, $\la'_j$,  such that 
\begin{equation}\label{la,la'=}
\la_j=\frac{n^{1/2}}{c}\log\frac{\tfrac{n^{1/2}}{c}}{h_j},\quad 
\la'_j=\frac{n^{1/2}}{c}\log\frac{\tfrac{n^{1/2}}{c}}{w_j}, \quad 1\le j\le k,
\end{equation}
$h_j$, $w_j$ (weakly) increase with $j$ increasing, and 
\begin{equation}\label{h,w;cond}
h_k,\,w_k=O(n^{\ga}\log n); \quad h_1,\,w_1\ge n^{-1/2+2\ga}.
\end{equation}

With $\nu$ instead of $n$ in Lemma \ref{pnrssim} we need to determine $h^*$ and $w^*$ such that
\[
r=\la_k-k=\left\lceil\frac{\nu^{1/2}}{c}\log\frac{\tfrac{\nu^{1/2}}{c}}{h^*}\right\rceil,\quad
s=\la'_k-k=\left\lceil\frac{\nu^{1/2}}{c}\log\frac{\tfrac{\nu^{1/2}}{c}}{w^*}\right\rceil,
\]
i.e
\begin{align*}
\frac{\nu^{1/2}}{c}\log\frac{\tfrac{\nu^{1/2}}{c}}{h^*}=&\,\frac{n^{1/2}}{c}\log\frac{\tfrac{n^{1/2}}{c}}{h_k}-k+O(1),\\
\frac{\nu^{1/2}}{c}\log\frac{\tfrac{\nu^{1/2}}{c}}{w^*}=&\,\frac{n^{1/2}}{c}\log\frac{\tfrac{n^{1/2}}{c}}{w_k}-k +O(1).
\end{align*}
Here, by the definition \eqref{nu=},
\begin{equation}\label{nu=appr}
\begin{aligned}
\nu=&\,n+k^2-\sum_{j=1}^k\frac{n^{1/2}}{c}\log\frac{\tfrac{n^{1/2}}{c}}{h_jw_j}+O(k)\\
=& n-\Theta\bigl(n^{\ga+1/2}\log n\bigr),
\end{aligned}
\end{equation}
$\Theta(A)$ denoting a remainder term of order $A$ exactly. Using the second line in 
\eqref{nu=appr}, {\it and\/} the fact that $|\log h_k|=\Theta(\log n)$, $|\log w_k|=\Theta(\log n)$, 
we easily show that $h^*$, $w^*$ exist, and are given by 
\begin{equation}\label{h*+w*}
h^*=\bigl(1+O(n^{\ga-1/2}\log^2 n)\bigr)h_k,\quad w^*=\bigl(1+O(n^{\ga-1/2}\log^2 n)\bigr)w_k,
\end{equation}
implying that
\begin{align*}
h^*+w^*=&\,h_k+w_k +O\bigl((h_k+w_k)n^{\ga-1/2}\log^2n\bigr)\\
=&\,h_k+w_k+O\bigl(n^{2\ga-1/2}\log^3n\bigr).
\end{align*}
Another simple evaluation, based on the first line in \eqref{nu=appr}, yields
\begin{equation}\label{pisqrt}
\pi\sqrt{\frac{2\nu}{3}}=\pi\sqrt{\frac{2n}{3}}-\sum_{j=1}^k\log\frac{\left(\tfrac{n^{1/2}}{c}\right)^2}{h_jw_j}+O\bigl(n^{2\ga-1/2}\log n\bigr).
\end{equation}
It follows then from Lemma \ref{pnrssim} and \eqref{h*+w*}, \eqref{pisqrt} that
\begin{equation}\label{pla,la'}
\begin{aligned}
p_n(\vec\la,\vec\la')=&\,\frac{e^{\pi\sqrt{2n/3}}}{4\sqrt{n}}\prod_{j=1}^k\frac{h_jw_j}{\left(\tfrac{n^{1/2}}{c}\right)^2}\\
&\times \exp\bigl(-h_k-w_k +O(n^{2\ga-1/2}\log^3n)\bigr).
\end{aligned}
\end{equation}
{\bf Note.\/} The RHS of \eqref{pla,la'} is obviously positive, whence it had better be  $1$, at least. 
And indeed, the RHS approaches infinity for $h_j,w_j$ meeting the constraints \eqref{h,w;cond}.
\begin{corollary}\label{Pr(la,la')=} Let $\gamma<1/4$. Then, uniformly for 
 $(\la_j,\la'_j)_{1\le j\le k}$ defined in \eqref{la,la'=}, with  $h_j,\,w_j$, ($1\le j \le [n^{\ga}]$) satisfying \eqref{h,w;cond}, 
\begin{align*}
&\pr\left(\bigcap\limits_{1\le j\le k}\{\La_j=\la_j, \La'_j=\la'_j\}\right)\\
&\quad=\exp\bigl(-h_k-w_k +O(n^{2\ga-1/2}\log^3n)\bigr)
\prod_{j=1}^k\frac{h_jw_j}{\left(\tfrac{n^{1/2}}{c}\right)^2}.
\end{align*}
\end{corollary}
To interpret this result gainfully, introduce the ``slanted'' $\varLambda_j$ as follows.
Let $\boldsymbol{\mathcal E}=(\mathcal E_1,\dots,\mathcal E_k)$, $\boldsymbol{\mathcal E}'=(\mathcal E_1',\dots,\mathcal E'_k)$ be such that all $\mathcal E_i,\,\mathcal E_i'$ are
independent copies of a random variable $\mathcal E$, with $\pr(\mathcal E>x)=e^{-x}$, $x\ge 0$. Define
\begin{align}
&\varLambda_j=\left\lceil\frac{n^{1/2}}{c}\log\frac{\tfrac{n^{1/2}}{c}}{S_j}\right\rceil,\quad 
\varLambda'_j=\left\lceil\frac{n^{1/2}}{c}\log\frac{\tfrac{n^{1/2}}{c}}{S'_j}\right\rceil, 
\label{slantLaj,La'j=}\\
&S_j=\sum_{i=1}^j\mathcal E_i,\quad S'_j=\sum_{i=1}^j\mathcal E_i'.\label{Hj,Wj=}
\end{align}
For the generic values $u_1,\dots u_k, v_1,\dots, v_k$ of $S_1,\dots, S_k,S'_1,\dots,S'_k$,
the joint density is $e^{-u_k-v_k}$, provided that $u_1\le \cdots \le u_k$, $v_1\le\cdots\le v_k$;
the density is zero otherwise. 

Let us compute $\pr\,\Bigl(\bigcap\limits_{1\le j\le k}\{\varLambda_j=\la_j,\varLambda'_j=\la'_j\}\Bigr)$.
Using  \eqref{la,la'=} and \eqref{slantLaj,La'j=}, we obtain that 
\begin{align*}
\varLambda_j=&\,\la_j\quad \text{ iff }\quad S_j\in I(\la_j):=\frac{n^{1/2}}{c}e^{-c\la_jn^{-1/2}}\bigl(e^{-cn^{-1/2}},1\bigr],\\
\varLambda'_j=&\,\la'_j\quad\text{ iff }\quad S'_j\in I(\la'_j):=\frac{n^{1/2}}{c}e^{-c\la'_jn^{-1/2}}\bigl(e^{-cn^{-1/2}},1\bigr].
\end{align*}
The intervals $I(\la_j)$ ($I(\la'_j)$ resp.) do not overlap, with $\inf\{x: x\in I(\la_j)\}
\ge \max\{x: x\in I(\la_{j-1})\}$ ($\inf\{x: x\in I(\la'_j)\}
\ge \max\{x: x\in I(\la'_{j-1})\}$ resp.), if $\la_{j -1}-\la_j\ge1$ ($\la'_{j-1}-\la'_j\ge 1$ resp.). For this
choice of $\la_1,\dots,\la_k,\la'_1,\dots,\la'_k$, we have
\begin{equation}\label{P(slLa=)=}
\pr\left(\bigcap\limits_{1\le j\le k}\!\!\!\{\varLambda_j=\la_j,\varLambda'_j=\la'_j\}\!\!\right)
=\prod_{j=1}^{k-1}\bigl|(\la_j)\bigr|\cdot \bigl|I(\la'_j)\bigr|\iint\limits_{u\in I(\la_k)\atop v\in I(\la'_k)}
\!\!\!e^{-u-v}\,du\,dv.
\end{equation}
For $1\le j\le k$,
\begin{multline*}
|I(\la_j)|=\frac{n^{1/2}}{c}e^{-c\la_j n^{-1/2}}\bigl(1-e^{-cn^{-1/2}}\bigr)\\
=\bigl(1+O(n^{-1/2})\bigr)\exp\left(-\log\frac{\tfrac{n^{1/2}}{c}}{h_j}\right)=(1+O(n^{-1/2}))\frac{h_j}
{\tfrac{n^{1/2}}{c}},
\end{multline*}
and likewise
\begin{equation*}
|I(\la'_j)|=\bigl(1+O(n^{-1/2})\bigr)\frac{w_j}{\tfrac{n^{1/2}}{c}}.
\end{equation*}
So 
\begin{equation}\label{j<k}
\prod_{j=1}^{k-1}\bigl|I(\la_j)\bigr| \cdot\bigl|I(\la'_j)\bigr|=(1+O(n^{\ga-1/2}))\prod_{j=1}^{k-1}\frac{h_jw_j}{\left(\tfrac{n^{1/2}}{c}\right)^2}.
\end{equation}
Next, the leftmost, i.e. the infimum, points of $I(\la_k)$ and $I(\la'_k)$ are respectively
\[
h_k\bigl(1+O(n^{-1/2})\bigr)=h_k+O\bigl(n^{\ga-1/2}\bigr),\quad 
w_k\bigl(1+O(n^{-1/2})\bigr)=w_k+O\bigl(n^{\ga-1/2}\bigr).
\]
Consequently
\begin{equation}\label{integ=}
\iint\limits_{u\in I(\la_k)\atop v\in I(\la'_k)}
\!\!\!e^{-u-v}\,du\,dv=\bigl(1+O(n^{\ga-1/2})\bigr)\frac{h_kw_k}{\left(\tfrac{n^{1/2}}{c}\right)^2}\,e^{-h_k-w_k}.
\end{equation}
Combining \eqref{P(slLa=)=}, \eqref{j<k} and \eqref{integ=}, we conclude
\begin{equation*}
\pr\left(\bigcap\limits_{1\le j\le k}\!\!\!\{\varLambda_j=\la_j,\varLambda'_j=\la'_j\}\!\!\right)
=(1+O(n^{\ga-1/2}))\prod_{j=1}^k\frac{h_jw_j}{\left(\tfrac{n^{1/2}}{c}\right)^2}e^{-h_k-w_k};
\end{equation*}
to remind, $\varLambda_j$, $\varLambda'_j$ are defined in \eqref{slantLaj,La'j=} and
\eqref{Hj,Wj=}. This equation and Corollary \ref{Pr(la,la')=} imply
\begin{lemma}\label{La=slanLa} Let $\mathcal G$ denote the set of all tuples
 $(\la_j,\la'_j)_{1\le j\le k}$ defined in \eqref{la,la'=}, with  $h_j,\,w_j$, ($1\le j \le [n^{\ga}]$) satisfying \eqref{h,w;cond}, and such that $\la_j$, $\la'_j$ strictly decrease with $j$. Uniformly
 for $(\la_j,\la'_j)_{1\le j\le k}\in \mathcal G$, we have
 \begin{equation}\label{PLambda=PsLambda}
 \begin{aligned}
 \pr\left(\bigcap\limits_{1\le j\le k}\{\La_j=\la_j, \La'_j=\la'_j\}\right)=&\,
 \left(1+O\bigl(n^{2\ga-1/2}\log^3n\bigr)\right)\\
 &\times \pr\left(\bigcap\limits_{1\le j\le k}\!\!\!\{\varLambda_j=\la_j,\varLambda'_j=\la'_j\}\!\!\right).
 \end{aligned}
 \end{equation}
\end{lemma}

\noindent Let us show in three steps that whp $(\varLambda_j,\varLambda'_j)_{1\le j\le k}\in \mathcal G$.\\

\noindent {\bf (1)\/} Recalling the definition of $\varLambda_j$, we have
\begin{align*}
&\pr\left(\min_{2\le j\le k}(\varLambda_{j-1}-\varLambda_j)=0\right)\le 
\sum_{j=2}^k \pr\bigl(\varLambda_{j-1}-\varLambda_j=0\bigr)\\
=&\sum_{j=2}^k\pr\left(\frac{S_j}{S_{j-1}}\le e^{cn^{-1/2}}\right)=\sum_{j=2}^k\,\,\,\,\iint\limits_{\tfrac{u+v}{u}
\le  \exp\left(\tfrac{c}{n^{1/2}}\right)}\frac{u^{j-2}
}{(j-2)!}e^{-u-v}\,dudv\\
=&\sum_{j=2}^k\int\limits_0^{\infty}\frac{u^{j-2}}{(j-2)!}\,e^{-u(1-e^{-c n^{-1/2}})^{-1}}\,du=
\sum_{j=2}^k(1-e^{-cn^{-1/2}})^{j-1}\\
&\qquad\qquad =\,O(k^2n^{-1/2})=O(n^{2\ga-1/2})\to 0.
\end{align*}
Thus, with probability $1-O\bigl(n^{2\ga-1/2}\bigr)$,  $\La_j$ strictly decreases, and similarly
so does $\La'_j$.\\

\noindent {\bf (2)\/}  Next, for $S_k$ in the formula for $\varLambda_k$ we have $\ex\bigl[e^{zS_k}\bigr]=(1-z)^{-k}$
if $z\in (0,1)$. So, using Chernoff-type bound,
\begin{equation*}
\pr\bigl(S_k\ge n^{\ga}\log n\bigr)\le\frac{\ex\bigl[e^{zS_k}\bigr]}{\exp\bigl(zn^{\ga}\log n\bigr)}
\le\frac{(1-z)^{-n^{\ga}}}{\exp\bigl(zn^{\ga}\log n\bigr)}.
\end{equation*}
The last fraction attains its minimum at $z=1-\tfrac{1}{\log n}$, and so
\begin{equation}\label{Sksmall}
\pr\bigl(S_k\ge n^{\ga}\log n\bigr)\le \exp\bigl(-0.5n^{\ga}\log n\bigr).
\end{equation}
Analogous bound holds for $S'_k$ in the formula for $\varLambda'_k$.\\

\noindent {\bf (3)\/} For $S_1$ in the formula for $\varLambda_1$, and $S'_1$ in the formula for $\varLambda'_1$, we have
\begin{equation}\label{S1,S'1}
\pr\bigl(S_1\le n^{-1/2+2\ga}\bigr)=\pr\bigl(S'_1\le n^{-1/2+2\ga}\bigr)=1-e^{-n^{-1/2+2\ga}}\le 
n^{-1/2+2\ga},
\end{equation}
a bound matching the bound in the item {\bf (1)\/}.\\

\noindent Putting together the items {\bf (1)\/}, {\bf (2)\/}, {\bf (3)\/}, and using $k\le n^{\gamma}$,  we obtain 
\begin{equation}\label{varLainG}
\pr\,\bigl((\varLambda_j,\varLambda'_j)_{1\le j\le k}\in \mathcal G)\ge 1- O\bigl(n^{2\ga-1/2}\bigr).
\end{equation}
 Adding up the equations \eqref{PLambda=PsLambda} for all $(\la_j,\la_j')_{1\le j\le k}\in
\mathcal G$ and using \eqref{varLainG}, we get
\begin{equation}\label{LainG}
\pr\,\bigl((\Lambda_j,\Lambda'_j)_{1\le j\le k}\in \mathcal G)=1-O\bigl(n^{2\ga-1/2}\log^3 n\bigr).
\end{equation}

Let $\mu_{(\vec\La,\vec\La')}$, $\mu_{(\vec\varLambda,\vec\varLambda')}$ denote the
probability distribution of $(\vec\La,\vec\La')$ and the probability distribution of $(\vec\varLambda,\vec\varLambda')$. Introduce $d_{TV}\bigl(\mu_{(\vec\La,\vec\La')},\mu_{(\vec\varLambda,\vec\varLambda')}\bigr)$, the {\it total variation\/} distance between the two distributions, i.e.
\[
d_{TV}\bigl(\mu_{(\vec\La,\vec\La')},\,\mu_{(\vec\varLambda,\vec\varLambda')}\bigr):=
\max_A\bigl|\pr\bigl((\vec\La,\vec\La')\in A\bigr)-\pr\bigl((\vec\varLambda,\vec\varLambda')\in A\bigr)\bigr|;
\]
here $A$ is a generic subset of the set of all decreasing tuples $(\la_j,\la'_j)_{1\le j\le k}$,
with $\sum_{1\le j\le k}\la_j\le n$,  $\sum_{1\le j\le k}\la'_j\le n$.
\begin{theorem}\label{dTVsmall} If $k=[n^{\ga}]$ and $\ga<1/4$, then 
\begin{equation*}
d_{TV}\bigl(\mu_{(\vec\La,\vec\La')},\,\mu_{(\vec\varLambda,\vec\varLambda')}\bigr)=
O\bigl(n^{2\ga-1/2}\log^3 n\bigr).
\end{equation*}
Consequently, for two generic subsets $B_1$, $B_2$ of all  decreasing tuples $(\la_j)_{1\le j\le k}$, with $\sum_{1\le j\le k}\la_j\le n$,
we have
\[
\pr(\vec\Lambda\in B_1, \vec\Lambda'\in B_2)=\pr(\vec\Lambda\in B_1)\pr(\vec\Lambda'\in B_2)+
O\bigl(n^{2\ga-1/2}\log^3 n\bigr).
\]
\end{theorem}
\begin{proof} Immediate, based on $A=(A\cap \mathcal G)\cup (A\cap \mathcal G^c)$ and \eqref{PLambda=PsLambda},
\eqref{varLainG},  \eqref{LainG}.
\end{proof}
{\bf Note.\/} Thus the probability of every event expressed through the heights of $k=[n^{\gamma}]$ tallest columns and the lengths of $k$ longest rows is within the distance $O\bigl(n^{2\ga-1/2}\log^3 n\bigr)$ from the probability of the event where the heights and the widths are replaced, respectively, with 
\begin{equation*}
\left\lceil \frac{n^{1/2}}{c}\log\frac{\tfrac{n^{1/2}}{c}}{\sum_{j=1}^i \mathcal E_j}\right\rceil_{1\le j\le k}\text{ and }\quad
\left\lceil \frac{n^{1/2}}{c}\log\frac{\tfrac{n^{1/2}}{c}}{\sum_{j=1}^i \mathcal E'_j}\right\rceil_{1\le i\le k}
\end{equation*}
In \cite{Pittel3} we proved a counterpart of this Theorem for the random $3$-dimensional diagram formed by packing $n$
unit cubes into a corner. For a {\it fixed\/} $k$: jointly, the $k$ largest lengths of vertical stacks of unit cubes, and the $k$ largest
lengths of the horizontal ``beams'' of unit cubes, parallel to the $x$-axis  and, separately, to the $y$-axis, converge in distribution to
a three-group version of the display above, with $n^{1/2}$ becoming $n^{1/3}$, and $c$ becoming $(2\zeta(3))^{1/3}$. Of course,
there is an additional sequence $\{\mathcal E''_j\}$ independent of $\{\mathcal E_j,\,\mathcal E_j'\}_{1\le j\le k}$.

\section{Graphical Partitions} Our task is to bound Wilf's probability $P(n):=\pr(\Lambda\text{ is graphical})$
using Theorem \ref{dTVsmall}. According
to the Nash-Williams conditions \eqref{Nash}, $\Lambda$ is graphical iff
\begin{equation}\label{N-W}
\sum_{j=1}^i\La'_j\ge \sum_{j=1}^i\La_j+i,\quad 1\le i\le D(\La),
\end{equation}
where $D(\La)$ is the size of the Durfee square of $\La$, i.e. the number of rows of the largest square inscribed into the Young diagram $\La$. 

Combining \eqref{Sksmall}, \eqref{S1,S'1} and Theorem \ref{dTVsmall}, we see that, with probability $1-O\bigl(n^{2\ga-1/2}\log^3n\bigr)$,
all $\La_j$, $\La'_j$ are between $a n^{1/2}\log n$ and $b n^{1/2}\log n$, for some
constants $a,b$. In particular, with probability this high,
$\La_k/k,\,\La'_k/k \gg 1$, whence $D(\Lambda)\gg k$,
and
\[
\sum_{j=1}^i\La_j+i=\bigl(1+O(n^{-1/2})\bigr)\sum_{j=1}^i\La_j.
\]
Using this fact and applying Theorem \ref{dTVsmall} yet again, we see that 
\begin{multline}\label{P(n)<expl}
P(n)\le O\bigl(n^{2\ga-1/2}\log^3n\bigr)\\
\qquad +\pr\left(\bigcap_{1\le u\le k}\!\!\left\{\frac{\varLambda'_u}{n^{1/2}\log n}\in [a,b]\right\}\bigcap_{i=1}^k\!\left\{\sum_{j=1}^i\varLambda_j\ge \bigl(1+O(n^{-1/2})\bigr)\!\sum_{j=1}^i\varLambda'_j\right\}\!\right)\!,
\end{multline}
where, as we recall,
\begin{align*}
\varLambda_j=&\left\lceil\frac{n^{1/2}}{c}\log\frac{\tfrac{n^{1/2}}{c}}{S_j}\right\rceil,\quad S_j=\sum_{\ell=1}^j \mathcal E_{\ell},\\
\varLambda'_j=&\,\left\lceil\frac{n^{1/2}}{c}\log\frac{\tfrac{n^{1/2}}{c}}{S'_j}\right\rceil,\quad S'_j=
\sum_{\ell=1}^j\mathcal E'_{\ell},
\end{align*}
with $\mathcal E_{\ell}$, $\mathcal E'_{\ell}$ being independent copies of $\mathcal E$, ($\pr(\mathcal E>x)=e^{=x}$). The
event on the RHS of \eqref{P(n)<expl} looks more complex than the Nash-Williams conditions
\eqref{N-W}, but crucially that event  is expressed in terms of eminently tractable sums of the
i.i.d. random variables.

Now, for 
\[
\frac{\varLambda'_u}{n^{1/2}\log n}\in [a,b], \quad 1\le u\le k,
\]
we have 
\[
\left\{\sum_{j=1}^i\varLambda_j\ge \bigl(1+O(n^{-1/2})\bigr)\!\sum_{j=1}^i\varLambda'_j\right\}
\subseteq\left\{\prod_{j=1}^i\frac{S'_j}{S_j}\ge 1+O\bigl(n^{\ga-1/2}\log n\bigr)\right\},
\]
uniformly for $i\le k$. Thus, for $k=[n^{\ga}]$,
\begin{equation}\label{P(n)<Pk}
\begin{aligned}
P(n)\le&\,O\bigl(n^{2\ga-1/2}\log^3n\bigr)+P_k(n),\\
P_k(n):=&\,\pr\left(\min_{1\le i\le k}\prod_{j = 1}^i\frac{S_{j}^{'}}{S_{j}} \ge \frac{1}{2}\right).
\end{aligned}
\end{equation}
That $\lim_{n\to\infty}P(n)=0$ was already proved in \cite{Pittel2}. We will show that in fact
$P(n)\to 0$ faster than any negative power of $\log n$.
To this end, we need to analyze the likely behavior of the products $\prod_{j = 1}^{i}S'_{j}/S_{j}$.
What follows is the series of estimates that ultimately will allow us to achieve this goal.
\\

{\bf (1)\/} For large $j$ the distributions of $S_j$ and $S_j'$ are concentrated around $j\ex\,[\mathcal E ]=j$. So it is natural to scale both $S_j$ and $S_j'$ by $j$ ant to center the resulting fractions by $1$. So we define $R_j=(S_j-j)/j$ and $R_j'=(S_j'-j)/j$, which are
equidistributed, of course. Using $S_j-j=\sum_{i=1}^j(\mathcal E_i-1)$, we have:  for $z\in (0,1)$,
\[
\ex\bigl[\exp(z(S_j-j))\bigr]=\left(\ex\bigl[\exp(z(\mathcal E-1))\bigr]\right)^j=\left(\frac{e^{-z}}{1-z}\right)^j.
\]
So, given $d>0$, we use the Chernoff-type bound to estimate
\[
\pr(R_j\ge d)=\pr\bigl(S_j-j\ge jd\bigr)
\le\frac{\left(\tfrac{e^{-z}}{1-z}\right)^j}{e^{zjd}}
=\exp\left[j\left(\log\frac{e^{-z}}{1-z}-zd\right)\right].
\]
The last function attains its {\em minimum} at $z=\tfrac{d}{1+d}$, and so
\begin{equation*}%\label{PmaxRj,jlarge}
\pr(R_j\ge d)\le \exp\bigl[j(\log(1+d)-d)\bigr]\le \exp\bigl(-jd^2/3\bigr),\quad\forall\, d\le 1/2.
\end{equation*}
Similarily
\begin{equation*}%\label{PmaxRj,jlarge'}
\pr(R_j\le -d)\le \exp\bigl[j(\log(1-d)+d)\bigr]\le \exp\bigl(-jd^2/2\bigr),\quad\forall\, d\le 1.
\end{equation*}
Therefore
\begin{equation}\label{PmaxRj,jlarge}
\pr(|R_j|\ge d)\le \exp\bigl[j(\log(1+d)-d)\bigr]\le 2\exp\bigl(-jd^2/3\bigr),\quad\forall\, d\le 1/2.
\end{equation}
Needless to say, we have the similar bound for $\pr(|R_j'|\ge d)$. \\

Let us use \eqref{PmaxRj,jlarge} to 
show that with high probability (whp) the $|R_j|$ are all small for $j$ at a sizable distance from $1$. Let $\chi=\chi(k)<k$ be such that  $\chi\to\infty$, $\chi=o(k)$ as $k\to\infty$.
Picking $\a>1/2$ and denoting $y_1=\log(\chi-1)$, $y_2=\log k$, we have
\begin{align*}
&\pr\left(\bigcup_{\chi\le j\le k}\bigl\{|R_j|\ge j^{-1/2}(\log j)^{\a}\bigr\}\right)\le \sum_{j=\chi}^k\pr\bigl(|R_j|\ge j^{-1/2}(\log j)^{\a}\bigr)\\
&\quad \le 2\sum_{j=\chi}^k e^{-(\log j)^{2\a}/3}\le 2\int\limits_{\chi-1}^k\!\!e^{-(\log x)^{2\a}/3}
\,dx
=2\!\int\limits_{y_1}^{y_2}\!\! e^{-y^{2\a}/3+y}\,dy\\
&\qquad\qquad\qquad\quad\text{using concavity of } -y^{2\a}/3+y\\
&\le2 e^{-y_1^{2\a}/3+y_1}\!\!\int\limits_{y_1}^{\infty}e^{(y-y_1)(1-2\a y_1^{2\a-1}/3)}\,dy
\le 2.5e^{-y_1^{2\a}+y_1}\le 3\chi e^{-(\log \chi)^{2\a}/3}.
\end{align*}
Therefore
\begin{equation}\label{expbound1}
\pr\left(\bigcup_{\chi\le j\le k}\bigl\{|R_j|\ge j^{-1/2}(\log j)^{\a}\bigr\}\right)\le
3\chi \exp\bigl(-(\log \chi)^{2\a}/3\bigr).
\end{equation}

{\bf (2)\/} Consider $j<\chi$. Let us prove that whp $|\sum_{j<\chi}(R_j'-R_j)|$ is not much larger than $\sqrt{\chi}$.
The variables $Y_j:=\mathcal E'_j-\mathcal E_j$ are all independent, ($\ex[Y_j]=0,\,\var (Y_j)=2$), and we have
%First of all,  denoting $R_j':=S'_j-j$,
%\[
%R_j^\prime-R_j=\frac{1}{j}\sum_{t=1}^jY_j,\quad Y_j:=\mathcal E_j^\prime-\mathcal E_j,
%\]
%So
\[
\sum_{j=1}^{\chi-1}(R_j^\prime-R_j)=\sum_{j=1}^{\chi-1}\frac{1}{j}\sum_{t=1}^j Y_t=
\sum_{t=1}^{\chi-1}Y_t\sum_{j=t}^{\chi-1}\frac{1}{j}.
\]
So, picking $\omega=\omega(k)\to\infty$, by Chebyshev's inequality, 
\begin{align}
\pr\left(\Bigl|\sum_{j=1}^{\chi}(R_j^\prime-R_j)\Bigr|>\sqrt{\omega\chi}\right)\le&\, (\omega\chi)^{-1}
\var\left(\sum_{j=1}^{\chi-1}(R_j^\prime-R_j)\right)\notag\\
=&\,2(\omega\chi)^{-1}\sum_{t=1}^{\chi-1}\left(\sum_{j=t}^{\chi-1}\frac{1}{j}\right)^2=O(\omega^{-1}),\label{sum,j<kappa}
\end{align}
since 
\begin{align*}
2\sum_{t=1}^{\chi-1}\left(\sum_{j=t}^{\chi-1}\frac{1}{j}\right)^2\le&\,\sum_{t=1}^{\chi-1}
2\left(\frac{1}{t}+\log\frac{\chi-1}{t}\right)^2
\le\sum_{t=1}^{\chi-1}\left(\frac{1}{t^2}+\log^2\frac{\chi-1}{t}\right)\\
\le&\,\sum_{t=1}^{\infty}\frac{1}{t^2}+(\chi-1)\int_0^1\log^2 x\,dx=O(\chi).
\end{align*}

{\bf (3)\/} Continuing with the case $j<\chi$, we need to obtain a likely {\it lower\/} bound for $\prod_{j<\chi}S_j/S_j'$.
Given $\be >1$, by independence of $S_j^\prime$ and $S_j$ we have: for $0<z <1$,
\begin{align*}
\pr\left(\frac{S_j^\prime}{S_j}\ge\be\right)=&\,\pr\bigl(e^{z(S_j^\prime-\be S_j)}\ge 1\bigr)\\
\le&\,\ex\bigl[e^{z(S_j^\prime-\be S_j)}\bigr]=\left(\frac{1}{1-z}\right)^j\left(\frac{1}{1+z\be}\right)^j.
\end{align*}
The function $(1-z)(1+z\be)$ attains its {\em maximum} at $z=\tfrac{\be-1}{2\be}\in (0,1)$, and
the maximum value is $1+\tfrac{(\be-1)^2}{4\be}$. Therefore
\[
\pr\left(\frac{S_j^\prime}{S_j}\ge\be\right)\le\left(1 +\frac{(\be-1)^2}{4\be}\right)^{-j}.
\]
For $\beta=\beta_j:=1+\sqrt{\tfrac{\omega}{j}}$, ($\omega=\omega(k)$ from the previous step {\bf (2)\/}), the bound is 
\[
1 +\frac{(\be-1)^2}{4\be}=f(\omega/j),\quad f(x):=1+\frac{1}{4}\frac{x}{1+\sqrt{x}}.
\]
\bi
Consequently
\[
\pr\left(\bigcup_{1\le j\le\chi-1}\left\{\frac{S_j^\prime}{S_j}\ge\be_j\right\}\right)\le\sum_{j=1}^{\chi-1}
\pr\left(\frac{S_j^\prime}{S_j}\ge\be_j\right)
\le\,\sum_{j=1}^{\chi-1}f(\omega/j)^{-j}.
\]
Let us bound the sum assuming, as we certainly can, that $\omega=\omega(k)\gg \chi$. In that case, for all $j\le \chi$, we have
$f(\omega/j)^{-j}\le \left(4\sqrt{j/\omega}\right)^{-j}$, and---by considering the ratio $\left(4\sqrt{(j+1)/\omega}\right)^{j+1}/ \left(4\sqrt{j/\omega}\right)^{j}$--- it follows that 
$
\sum_{j=1}^{\chi-1}f(\omega/j)^{-j}\le 5\omega^{-1/2}.
$
So, denoting $E_1:=\bigcup_{1\le j<\chi}\left\{\frac{S_j^\prime}{S_j}\ge 1+\sqrt{\frac{\omega}{j}}\right\}$, we conclude that
\begin{equation}\label{expbound2}
\pr(E_1)\le 5\omega^{-1/2}. %\quad E_1:=\bigcup_{1\le j<\chi}\left\{\frac{S_j^\prime}{S_j}\ge 1+\sqrt{\frac{\omega}{j}}\right\}.
\end{equation}
%Consider $j\ge\omega$.  Using $1+x\ge \exp\bigl(x - \tfrac{x^2}{2(1-x)}\bigr)$
%for $x<1$, 
%\[
%f(\omega/j)\ge 1+\frac{\omega/j}{8}\ge \exp\bigl(\tfrac{\omega}{9j}\bigr),
%\]
%so
%\[
%\sum_{j\in [\omega,\chi]}f(\omega/j)^{-j}\le \chi e^{-\omega/9}.
%\]
%For $j\le\omega$, introducing $y_j=j/\omega$,
%\[
%f(\omega/j)^{-j}\le [g(y_j)]^{\omega}, \quad g(y):=
%\left(1+\frac{1}{8\sqrt{y}}\,\right)^{-y}.
%\]
%Since $(\log g(y))^{\prime\prime}>0$, 
%\begin{align*}
%\max_{y\in [1/\omega,1]}g(y)=&\max\{g(1/\omega),g(1)\}\\
%=&\,\max\left\{\!\!\left(1+\frac{\sqrt{\omega}}{8}\right)^{-1/\omega}\!\!,\frac{8}{9}\right\}
%=\left(1+\frac{\sqrt{\omega}}{8}\right)^{-1/\omega}\!\!\!,
%\end{align*}
%for all $\omega$ large enough. Therefore
%\[
%f(\omega/j)^{-j}\le f(\omega)=\left(1+\frac{\sqrt{\omega}}{8}\right)^{-1},
%\]
%and it can be shown that
%\[
%\sum_{j\in [1,\omega]}f(\omega/j)\le 2f(\omega)=\frac{2}{1+\frac{\sqrt{\omega}}{8}}\le 16\,\omega^{-1/2}.
%\]
%We conclude that
%\begin{equation}\label{expbound2}
%\pr\left(\bigcup_{1\le j<\chi}\left\{\frac{S_j^\prime}{S_j}\ge 1+\sqrt{\frac{\omega}{j}}\right\}\right)
%\le \chi e^{-\omega/9}+16\,\omega^{-1/2}.%\le 17\,\omega^{-1/2},
%\end{equation}
%if $\omega =10\log \chi$.

{\bf (4)\/} For $P_k(n)$ defined in \eqref{P(n)<Pk} we obviously have
\[
P_k(n)\le P_{k,1}(n):=\pr\left\{\min_{\chi\le \ell\le k}\prod_{j = \chi}^{\ell}\frac{S_{j}^{'}}{S_{j}} \ge \frac{1}{2}\prod_{j = 1}^{\chi -1}\frac{S_{j}}{S_{j}^{'}}\right\}.
\]
To find an explicit bound for $P_{k,1}(n)$, we need to find a larger {\it analyzable\/} event by replacing the LHS and the RHS of the event in question with their {\it sufficiently likely\/} upper bound and lower bound, respectively.  

Starting with the RHS, on the likely event $E_1^c$ %:=\bigcap_{j<\chi}\left\{\frac{S_j^\prime}{S_j}<
%1+\sqrt{\frac{\omega}{j}}\right\}$ 
we have
\begin{align*}
\prod_{j <\chi}\tfrac{S_{j}}{S_{j}^{'}}\ge&\, \prod_{j<\chi}\left(1+\sqrt{\frac{\omega}{j}}\,\right)^{-1}
\ge\,\exp\left(-\sum_{j<\chi}\sqrt{\frac{\omega}{j}}\right)
\ge\exp\bigl(-3\sqrt{\omega \chi}\bigr).
\end{align*}
Turning to the LHS, on the likely event 
\[
E_2:=\Biggl\{\sum_{j=1}^{\chi-1}(R_j-R_j')\le \sqrt{\omega \chi}\Biggr\}\bigcap\bigcup_{\chi\le j\le k}\!\!\bigl\{|R_j|+|R_j^\prime|\le2 j^{-1/2}(\log j)^{\alpha}\bigr\},
\]
 by the inequality $\log(1+x)\le x$, we have:
\begin{align}
\prod_{j=\chi}^{\ell}\frac{S_j^\prime}{S_j}=&\,\prod_{j=\chi}^{\ell}\frac{1+R_j^\prime}
{1+R_j}\le\exp\left(\sum_{j=\chi}^{\ell}\log\frac{1+R_j^\prime}{1+R_j}\right)\le
\exp\left(\sum_{j=\chi}^{\ell}\frac{R_j^\prime-R_j}{1+R_j}\right)\notag\\
&\quad\le\,\exp\left(\sum_{j=\chi}^{\ell}(R_j^\prime-R_j)+4\sum_{j=\chi}^{\ell}j^{-1}(\log j)^{2\a}\right)\notag\\
&\quad\le\,\exp\left(\sum_{j=1}^{\ell}(R_j^\prime-R_j)+\sqrt{\omega \chi}+2(\log k)^{2\a+1}\right).\label{sum,j>kappa}
\end{align}
Combining \eqref{expbound1}, \eqref{sum,j<kappa}, \eqref{expbound2} and \eqref{sum,j>kappa},
we arrive at 
\begin{equation}\label{Pk<Pk2}
\begin{aligned}
&\quad\qquad P_k(n)\le P_{k,1}(n)\le  P_{k,2}(n) +U,\\
&P_{k,2}(n):=\,\pr\left\{\min_{\chi\le\ell\le k}\sum_{j=1}^{\ell}(R_j^\prime-R_j)\ge -V\right\},
\\
&U=U(\chi,\omega):=\,6\chi \exp\bigl(-(\log \chi)^{2\a}/3\bigr)+6\, \omega^{-1/2},\\
&\quad\,\, V=V(\chi,\omega,k):=\,4\sqrt{\omega \chi}+2(\log k)^{2\a+1};
\end{aligned}
\end{equation}
here $\alpha>1/2$ is fixed. We will use \eqref{Pk<Pk2} for
\begin{equation}\label{kappa,omega=}
\chi=\left\lfloor\exp\left((\log k)^{\tfrac{1}{2\alpha}}\right)\right\rfloor,\quad \omega=\exp\left(
\frac{2\log k}{\log\log k}\right),
\end{equation}
(i.e. $\omega\gg k$ as stipulated earlier). For these $\chi,\,\omega$, we have 
\begin{equation}\label{U,V=}
\begin{aligned}
U\le U^*:=&\,7\exp\left(-\frac{\log k}{\log\log k}\right),\\
 V\le V^*:=&\,\exp\left(\frac{2\log k}{\log\log k}\right).
 \end{aligned}
 \end{equation}
That these $\chi$ and $\omega$ work will become clearer later.\\

{\bf (5)\/} Let us use independence of $Y_j=\mathcal E_j^\prime-\mathcal E_j$ to show that the event in \eqref{Pk<Pk2}
is {\it almost\/} (i.e. aside from an event of negligible probability) contained in the intersection of a sequence of independent events such that the product of their
individual probabilities goes to zero at a certain {\em explicit} rate. Introduce the sequence
$\{\ell_r\}$: pick $a>0$ and define
\begin{equation}\label{ellrdef}
\ell_0=\chi,\quad \ell_r=\ell_0\zeta^r,\quad \zeta=\zeta(k)=\lfloor a(\log\log k)\log k\rfloor,
\end{equation}
That is, $\ell_r$ form a geometric progression with denominator $\zeta$. So 
\begin{equation}\label{rho(k)}
|\{r>0: \ell_r\le k\}| \ge \rho=\rho(k):=\left\lfloor\frac{\log\tfrac{k}{\chi}}
{\log\bigl[a(\log\log k)\log k\bigr]}\right\rfloor\sim \frac{\log k}{\log\log k},
\end{equation}
since by \eqref{kappa,omega=} $\log\chi=o(\log k)$. Introducing 
\[
Z_{r}:=\sum_{j=1}^{\ell_r}(R_j^\prime-R_j)=\sum_{j=1}^{\ell_r}\frac{1}{j}\sum_{t=1}^jY_t,\qquad
Y_t=E_t^\prime-E_t,
\]
we have: for $r\in [1,\rho]$,
\begin{align}
Z_r=&\,\sum_{j=1}^{\ell_r}\frac{1}{j}\sum_{t=1}^jY_t=\sum_{t=1}^{\ell_r}Y_t
\sum_{j=t}^{\ell_r}\frac{1}{j}\notag\\
=&\sum_{t=1}^{\ell_{r-1}}Y_t
\sum_{j=t}^{\ell_r}\frac{1}{j}+\sum_{t=\ell_{r-1}+1}^{\ell_r}Y_t
\sum_{j=t}^{\ell_r}\frac{1}{j}=:Z_{r,1}+Z_{r,2}.\label{Zr=}
\end{align}
Here $Z_{r,1}$ is measurable with respect to $\mathcal F_{r-1}$, the $\sigma$-field generated by $Y_1,\dots,Y_{\ell_{r-1}}$, and $Z_{r,2}$ is independent of  $\mathcal F_{r-1}$. Crucially as well,
$Z_{r,2}$ are mutually independent.

To see the intuitive reason behind our choice of $\ell_r$,  observe that, since $Y_t$ are independent and $\ex[Y_t]=0$, $\var(Y_t)=2$, we have $\ex[Z_{r,1}]=0$ and 
\begin{equation}\label{VarZr1<}
\begin{aligned}
&\var(Z_{r,1})=2\sum_{t=1}^{\ell_{r-1}}\left(\sum_{j=t}^{\ell_r}\frac{1}{j}\right)^2
\le\sum_{t=1}^{\ell_{r-1}}\left(\frac{1}{t^2}+\log^2\frac{\ell_r}{t}\right)\\
\le&\, \frac{\pi^2}{6}+\ell_r\!\!\int\limits_0^{\tfrac{\ell_{r-1}}{\ell_r}}\!\!(\log x)^2\,dx
\le (2+\varepsilon_r)\ell_{r-1}\left(\log\frac{\ell_r}{\ell_{r-1}}\right)^2,\quad (\varepsilon_r\to 0),
\end{aligned}
\end{equation}
i.e. $\var(Z_{r,1})=O\bigl(\ell_{r-1}(\log\log k)^2\bigr)$. Thus, {\em typically}, $|Z_{r,1}|$ is roughly of order $O\bigl(\sqrt{\ell_{r-1}}\log\log k\bigr)$.  As for $Z_{r,2}$, its expected value is also zero, and
\begin{equation}\label{VarZr2sim}
\begin{aligned}
\var(Z_{r,2})=&\,2\sum_{t=\ell_{r-1}+1}^{\ell_r}\left(\sum_{j=t}^{\ell_r}\frac{1}{j}\right)^2
\ge 2\sum_{t=\ell_{r-1}+1}^{\ell_r}\log^2\left(\frac{\ell_r}{t}\right)\\
\sim&\,2\ell_r\int_0^1(\log x)^2\,dx=4\ell_r,
\end{aligned}
\end{equation}
meaning, {\em hopefully}, that $|Z_{r,2}|$ assumes values of order $\sqrt{\ell_r}$ with not too small probability. And $[-\sqrt{\ell_r},\sqrt{\ell_r}]\ni 
[-\sqrt{\ell_{r-1}}\log\log k,\sqrt{\ell_{r-1}}\log\log k]$ with room to spare, because 
\[
\frac{\sqrt{\ell_{r-1}}\log\log k}{\sqrt{\ell_r}}=O\left(\sqrt{\frac{\log\log k}{\log k}}\,\right).
\]
Notice also that $\sqrt{\ell_r}\gg V^*$, which is the bound for $V$ defined in \eqref{U,V=},  if 
\begin{equation}\label{r*=}
r\ge r^*:=\frac{3\log k}{(\log\log k)^2}\ll \rho(k)\,\sim\frac{\log k}{\log\log k}.
\end{equation}
So we should expect that for those $r$,  conditioned on $\mathcal F_{r-1}$, the probability that $Z_r>-V^*$ is close to the unconditional $\pr(Z_{r,2}>-V^*)\sim \tfrac{1}{2}$, the latter holding because $Z_{r,2}$ and $-Z_{r,2}$ are 
equidistributed. Once these steps are justified, we will get a bound analogous to, but weaker
than the naive bound
\begin{equation*}%\label{thebound}
P_{k,2}(n)\le \left(\frac{1}{2}+o(1)\right)^{\rho(k)}\!\!=\exp\bigl(-(\log 2)\rho(k)\bigr),
\end{equation*}
see \eqref{rho(k)} for the definition of $\rho(k)$.\\

To start, since $\ex\bigl[e^{\xi Y}\bigr]=(1-\xi^2)^{-1}$ for $|\xi|<1$, we see that, by \eqref{VarZr1<},
\begin{align*}
&\qquad\qquad\qquad\ex\bigl[\exp(\xi Z_{r,1})\bigr]=\,\prod_{t=1}^{\ell_{r-1}}\left[1-\xi^2\left(\sum_{j=t}^{\ell_r}\frac{1}{j}\right)^2\right]^{-1}\\
&\qquad\qquad\qquad\qquad\text{for }|\xi|\ll \Bigl(\sum_{j=1}^{\ell_r}1/j\Bigr)^{-1}\sim \log^{-1}\ell_r\\
&\le\exp\left[\xi^2(1+o(1))\sum_{t=1}^{\ell_{r-1}}\left(\sum_{j=t}^{\ell_r}\frac{1}{j}\right)^2\right]
= \exp\left[(2+o(1))\ell_{r-1}\left(\log\frac{\ell_r}{\ell_{r-1}}\right)^2\xi^2\right].
\end{align*}
%for $|\xi|=o\bigl((\log\ell_r)^{-1}\bigr)$.
%\[
%|\xi|\le\frac{1}{2}\left(\sum_{j=1}^{\ell_r}\frac{1}{j}\right)^{-1},
%\]
%i.e. certainly for $|\xi|\le \tfrac{1}{2}(1+\log\ell_r)^{-1}$. 
Consequently, for every such $\xi>0$,
\[
\pr\bigl(Z_{r,1}\ge \sqrt{\ell_r}\bigr)\le \exp\left[(2+o(1))\ell_{r-1}\left(\log\frac{\ell_r}{\ell_{r-1}}\right)^2\xi^2-\xi\sqrt{\ell_r}\right].
\]
The exponent attains its absolute minimum at 
\[
\xi=\frac{\sqrt{\ell_r}}{(2+o(1))\ell_{r-1}\left(\log\tfrac{\ell_r}{\ell_{r-1}}\right)^2}\ll \zeta^{-(r-2)/2}\ll
(\log k)^{-1} =O\bigl((\log \ell_r)^{-1}\bigr),
\]
for  $r\ge 4$, as $\ell_r\le k$, meaning that this $\xi$ is (easily) admissible for $r\ge r^*$.  For these $r$ and  $\xi$ we obtain
\begin{align*}
\pr\bigl(Z_{r,1}\ge \sqrt{\ell_r}\bigr)\le&\, \exp\Biggl(-\,\frac{\ell_r}{(4+o(1))\ell_{r-1}\left(\log\tfrac{\ell_r}{\ell_{r-1}}\right)^2}\Biggr)\\
&\,= \exp\left(\,-\frac{a\log k}{(4+o(1))\log\log k}\right),
\end{align*}
if $k$ is large enough; see \eqref{ellrdef} for the ratio $\ell_r/\ell_{r-1}$.  Consequently, with $\rho=\rho(k)$ defined in \eqref{rho(k)},
\begin{equation}\label{P(Zr1small)}
\begin{aligned}
\pr\left(\bigcup_{r=r^*}^{\rho}\left\{Z_{r,1}\ge\,\sqrt{\ell_r}\right\}\right)\le&\, \rho\exp\left(-\frac{a\log k}{(4+o(1))\log\log k}\right)\\
=&\, \exp\left(-\frac{a\log k}{(4+o(1))\log\log k}\right).
\end{aligned}
\end{equation}
Notice at once that the last RHS dwarfs $U^*$, see \eqref{U,V=}, the upper bound
for the term $U$ in the inequality \eqref{Pk<Pk2} for $P_k$.

The rest is short. Since the random variables $Z_{r,1}$ are mutually independent, we use
\eqref{P(Zr1small)} to bound
\begin{align*}
P_{k,2}(n)\le&\,\pr\left(\bigcap_{r=r^*}^{\rho}\{Z_r\ge -V^*\}\right)\\
\le&\,\pr\left(\bigcup_{r=r^*}^{\rho}\left\{Z_{r,1}\ge\,\sqrt{\ell_r}\right\}\right)+
\pr\left(\bigcap_{r=r^*}^{\rho}\{Z_r\ge -V^*,\,Z_{r,1}\le \sqrt{\ell_r}\}\right)\\
\le&\,\exp\left(-\frac{a\log k}{(4+o(1))\log\log k}\right)+\prod_{r=r^*}^{\rho}\pr\bigl(Z_{r,2}\ge -V^*-\sqrt{\ell_r}\bigr).
\end{align*}
Now, $r^*$ was chosen to make $\sqrt{\ell_r}\gg V^*$ (see the line preceding \eqref{r*=}), and (by \eqref{VarZr2sim}) we also know
that $\var(Z_{r,2})\sim 4\ell_r$. Since $Z_{r,2}$ is asymptotically normal with mean $0$ and
variance $4\ell_r$ we see that
\[
\pr\bigl(Z_{r,2}\ge -V^*-\sqrt{\ell_r}\bigr)=\frac{1}{\sqrt{2\pi}}\int_{-1/2}^{\infty}e^{-u^2/2}\,du +o(1).
\]
Since $\rho=\rho(k)\sim\tfrac{\log k}{\log\log k}$, and $r^*=o(\rho)$, we obtain
\begin{multline}\label{Pk2<exp(-alogk)}
P_{k,2}(n)\le \exp\left(-\frac{a\log k}{(4+o(1))\log\log k}\right)\\
+\left(\frac{1}{\sqrt{2\pi}}\int_{-1/2}^{\infty}e^{-u^2/2}\,du +o(1)\right)^{\tfrac{\log k}{\log\log k}}
\le\exp\left(-\frac{b\log k}{\log\log k}\right),
\end{multline}
if 
\[
b<\log \left(\frac{1}{\sqrt{2\pi}}\int_{-1/2}^{\infty}e^{-u^2/2}\,du\right)^{-1}=0.445...,
\]
and $a>4b$. Observe also that
\[
\exp\left(-\frac{b\log k}{\log\log k}\right)\gg U^*:=7\exp\left(-\tfrac{\log k}{\log\log k}\right),
\]
see \eqref{U,V=}. So, by \eqref{Pk<Pk2},
\begin{equation*}\label{thebound}
P_k(n)\le P_{k,2}(n)+U^*\le \exp\left(-\frac{0.445..\log k}{\log\log k}\right).
\end{equation*}
Using this bound and \eqref{P(n)<Pk}, we conclude that, for every $\ga<1/4$,
\[
P(n)\le O\bigl(n^{2\ga-1/2}\log^3 n\bigr)+\left.\exp\left(-\frac{0.445..\log k}{\log\log k}\right)\right|_{k=[n^{\ga}]}.
\]
\begin{theorem}\label{P(n)bound} For $n$ large enough,
\[
P(n)\le \exp\left(-\frac{0.11\log n}{\log\log n}\right).
\]
\end{theorem}

\section{Pairs of comparable partitions} Almost as an afterthought, let $\Lambda$, $\Delta$ be two uniformly random,
mutually independent, partitions of $n$. Our task is to bound Macdonald's probability, i.e.
\[
Q(n):=\pr(\Lambda\succeq\Delta)=\pr\left(\bigcap_i\left\{\sum_{j=1}^i\Delta_j\le 
\sum_{j=1}^i\Lambda_j\right\}\right).
\]
(Notice that the Durfee squares do not enter the comparability conditions at all.)
This time the tuples $(\Delta_j)_{j\ge 1}$ and $(\Lambda_j)_{j\ge 1}$ are mutually independent.
 Furthermore, the distributions of the sub-tuples,
$(\Delta_j)_{1\le i\le [n^{\ga}]}$ and $(\Lambda_j)_{1\le i\le [n^{\ga}]}$, are each within the total variation distance of order $O\bigl(n^{-1/2+2\ga}$ $\log^3n\bigr)$ from 
the distribution of $(\varLambda_j)_{1\le i\le [n^{\ga}]}$; see Theorem \ref{dTVsmall}. Thus
a simplified version of the proof of Theorem \ref{P(n)bound} establishes
\begin{theorem}\label{Q(n)bound} For $n$ large enough,
\[
Q(n)\le \exp\left(-\frac{0.11\log n}{\log\log n}\right).
\]
\end{theorem}
\bi

\noindent {\bf Acknowledgment.\/} I learned about
Herb Wilf's conjecture from a talk  on an efficient algorithm for
counting graphical partitions given by Carla Savage during the 1996 Conference honoring Herb. Later
Bruce Richmond encouraged me to see whether his work with Paul Erd\"os and the follow-up study by Cecil Rousseau
and Firasath Ali could be extended to yield a resolution of the conjecture.  Indeed, beneath the surface
of asymptotics there was an infinite sequence of independent random variables allowing to prove that the fraction of graphical partition
goes to zero, but leaving the rate question open.
Since then I intermittently contemplated using the full strength of that approach
in order to bound the convergence rate.
A thoughtful letter from Bruce pushed me to get it done finally. I owe a debt of gratitude to him for the inspiring feedback.
I thank the referee for expert reading of the manuscript and helpful suggestions.

\end{document}